\documentclass{amsart}
\usepackage{amsmath}    
\usepackage{graphicx}   
\usepackage{verbatim}   
\usepackage{amssymb}
\usepackage{amsthm}
\usepackage{mathrsfs}
\usepackage{bbm}
\usepackage{exscale}
\newtheorem{theorem}{Theorem}[section]
\newtheorem{lemma}[theorem]{Lemma}

\newtheorem{proposition}[theorem]{Proposition}

\newtheorem{definition}[theorem]{Definition}

\theoremstyle{definition}
\newtheorem{remark}[theorem]{Remark}
\newcommand{\finsum}[3]{
\underset{#1=#2}{\overset{#3}\sum}}

\newcommand{\dint}{
\displaystyle\int}

\newcommand{\dsum}{
\displaystyle\sum}

\newcommand{\dpifrac}{
\dfrac{1}{(2\pi)^\frac{d}{2}}}

\newcommand{\inflim}[1]{
\underset{#1\to\infty}\lim}

\newcommand{\nsum}[1]{
\underset{#1\in\N}\dsum}

\newcommand{\R}{
\mathbb{R}}

\newcommand{\N}{
\mathbb{N}}

\newcommand{\F}{
\mathscr{F}}

\newcommand{\bracket}[1]{
\left\langle#1\right\rangle}

\newcommand{\I}{
\mathscr{I}}

\newcommand{\phihat}{
\widehat{\phi}}
\newcommand{\phiahat}{
\widehat{\phi_\alpha}}

 \newcommand{\product}[3]{
\underset{#1=#2}{\overset{#3}\prod}}

\allowdisplaybreaks

\title{Nonuniform Sampling and Recovery of Bandlimited Functions in Higher Dimensions}
\author{Keaton Hamm}
\thanks{This work is part of the author's doctoral dissertation.  He thanks his advisors Th. Schlumprecht and N. Sivakumar for their guidance.  The author also takes pleasure in thanking the anonymous referees whose valuable comments greatly improved this paper.  The research was partially supported by National Science Foundation grant DMS 1160633}

\address{Department of Mathematics,Vanderbilt University, Nashville, Tennessee, USA}
\email{keaton.hamm@vanderbilt.edu} 

\begin{document}

\begin{abstract}
We provide sufficient conditions on a family of functions $(\phi_\alpha)_{\alpha\in A}:\R^d\to\R$ for sampling of multivariate bandlimited functions at certain nonuniform sequences of points in $\R^d$.  We consider interpolation of functions whose Fourier transform is supported in some small ball in $\R^d$ at scattered points $(x_j)_{j\in\N}$ such that the complex exponentials $\left(e^{-i\bracket{x_j,\cdot}}\right)_{j\in\N}$ form a Riesz basis for the $L_2$ space of a convex body containing the ball.  Recovery results as well as corresponding approximation orders in terms of the parameter $\alpha$ are obtained.
\end{abstract}

\maketitle
 \allowdisplaybreaks


\section{Introduction}

The theory of interpolation has long been of interest to approximation theorists, and has connections with many areas of mathematics including harmonic analysis, signal processing, and sampling theory (to name just a few).  The theory of spline interpolation at the integer lattice was championed by I.J. Schoenberg, and typically falls under the heading of ``cardinal spline interpolation.''  More generally, cardinal interpolation schemes are those in which a given target function is interpolated at the multi-integer lattice in $\R^d$.  Some study has been made of the connection between cardinal interpolation, sampling theory of bandlimited functions, and radial basis function theory.  Schoenberg himself showed that bandlimited functions can be recovered by their cardinal spline interpolants in a limiting sense as the order of the spline tends to infinity, and similar analysis shows that such functions can also be recovered by cardinal Gaussian and multiquadric interpolants.

Lately, the ideas of Schoenberg and many of his successors have been used to tackle problems in a broader setting, namely interpolation schemes at infinite point-sets that are nonuniform.  The richness of the theory for lattices suggests a search for comparable results in the nonuniform setting. To that end, Lyubarskii and Madych \cite{lm} considered univariate bandlimited function interpolation and recovery by splines, thus extending Schoenberg's ideas to the nonuniform setting.  Inspired by their work, Schlumprecht and Sivakumar \cite{ss} showed analogous recovery results using translates of the Gaussian kernel.  Then Ledford \cite{ledford_scattered} gave sufficient conditions on a family of functions to yield similar convergence results for bandlimited functions in one dimension.  One of the unifying themes in these works is the use of a special structure on the points, namely that they form {\em Riesz-basis sequences} (or complete interpolating sequences) for an associated Paley-Wiener space, or equivalently, the corresponding sequence of complex exponential functions forms a Riesz basis for a certain $L_2$ space.  We will discuss this in more detail later, but the main point here is that in one dimension, such sequences are characterized and relatively easy to come by.  However, in higher dimensions, the problem becomes significantly more complicated as the existence of such Riesz-basis sequences is unknown even for nice domains such as the Euclidean ball.

Some steps toward higher dimensional recovery schemes have been taken by Bailey, Schlumprecht, and Sivakumar \cite{bss} and Ledford \cite{ledford_bivariate}.  The former consider Gaussian interpolation of bandlimited functions whose band lies in a ball of small radius $\beta$, where the interpolation is done at a Riesz-basis sequence for some larger symmetric convex body.  It is known that for certain types of these bodies, namely {\em zonotopes}, there are Riesz-basis sequences in any dimension.  Moreover, one can approximate the Euclidean ball arbitrarily closely with such bodies.  That is, given any $\delta<1$, one can find a zonotope that lies between the ball of radius $\delta$ and the unit ball, and there is an associated Riesz-basis sequence for the Paley-Wiener space over the zonotope.  Ledford worked with squares in two dimensions using Poisson kernels for the interpolation scheme, and mentions an extension to cubes in higher dimensions.  Owing to the geometry of the problem, the use of cubes requires a careful analysis, when interpolating with radial functions.  It seems that with the techniques available, the geometry best suited to bandlimited function interpolation in higher dimensions is that of Paley-Wiener spaces over balls.  In fact, the main theorem in \cite{ss} holds in higher dimensions for functions whose band lies in the unit ball, but as mentioned above, this may well be vacuous if there is no Riesz-basis sequence for that space.  

Inspired by Ledford's conditions for univariate interpolation in \cite{ledford_scattered} and the higher dimensional Gaussian interpolation results in \cite{bss}, we give sufficient conditions on a family of functions to form interpolants for the Paley-Wiener space associated with some symmetric convex body in $\R^d$, such as a zonotope, which also provides recovery of bandlimited functions whose Fourier transforms are supported in a ball contained in the convex body.

The rest of the paper is laid out as follows.  We begin with  basic notations and facts in Section \ref{SECbasic}; we discuss our interpolation scheme and state the main theorem in Section \ref{SECinterp}.  Section \ref{SECexamples} provides several examples of families of functions that can be used for interpolation, while Section \ref{SECproofs} contains the proofs of the main results.  The paper concludes with a series of remarks.

\section{Basic Notions}\label{SECbasic}

If $\Omega\subset\R^d$ has positive Lebesgue measure, then let $L_p(\Omega)$, $1\leq p\leq\infty$, be the usual Lebesgue space over $\Omega$ with its usual norm.  If no set is specified, we mean $L_p(\R^d)$.  Similarly, let $\ell_p:=\ell_p(\N)$ be the usual sequence spaces indexed by the natural numbers.

Let $C(\R^d)$ be the space of all continuous functions on $\R^d$, and 
$$C_0(\R^d):=\{f\in C(\R^d):\underset{\|x\|\to\infty}{\lim}f(x) = 0\},$$
where $\|\cdot\|$ is the Euclidean distance on $\R^d$.  

The Fourier transform of an integrable function $f$ is given by
\begin{equation}\label{ft}
 \widehat{f}(\xi):=\dpifrac\int_{\R^d} f(x)e^{-i\bracket{\xi, x}}dx,\quad \xi\in\R^d,
\end{equation}
where $\bracket{\cdot,\cdot}$ denotes the usual scalar product on $\R^d$.  The Fourier transform can also be extended from $L_1\cap L_2$ to an isometry of $L_2$ onto itself.  We will denote by $\F[f]$, the Fourier transform of a function $f\in L_2$.  
Moreover, Parseval's Identity states that
\begin{equation}\label{parseval}
\|\F[f]\|_{L_2} = \|f\|_{L_2} .
\end{equation}
If $f$ is also continuous, and $\F[f]\in L_1$, then the following inversion formula holds:
\begin{equation}\label{EQfourierinversion}
 f(x)=\dpifrac\int_{\R^d}\F[f](\xi)e^{i\bracket{\xi,x}}d\xi,\quad x\in\R^d.
\end{equation}

We consider interpolation of so-called {\em bandlimited} or {\em Paley-Wiener functions}.  Precisely, for a bounded measurable set $S\subset\R^d$ with positive Lebesgue measure, we define the associated Paley-Wiener space to be
$$PW_S:=\left\{f\in L_2:\F[f] = 0 \textnormal{ a.e. outside of } S\right\}.$$
Consequently, for $f\in PW_S$, the inversion formula and the Riemann-Lebesgue Lemma imply that
$$f(x) = \dpifrac\int_{\R^d}\F[f](\xi)e^{i\bracket{\xi,x}}d\xi = \dpifrac\int_S\F[f](\xi)e^{i\bracket{\xi,x}}d\xi,$$
for every $x\in\R^d$, and moreover that $f\in C_0(\R^d)$.

A sequence of functions, $(\phi_j)_{j\in\N}$, is said to be a Riesz basis for a Hilbert space $\mathcal{H}$ if every element $h\in\mathcal{H}$ has a unique representation
\begin{equation}\label{EQrbasis}
 h=\nsum{j}a_j\phi_j,\quad \nsum{j}|a_j|^2<\infty,
\end{equation}
and consequently by the Uniform Boundedness Principle, there exists a constant $R_b$, called the basis constant, such that
\begin{equation}\label{EQRieszbasisconstant}
 \dfrac{1}{R_b}\left(\nsum{j}|c_j|^2\right)^\frac{1}{2}\leq\left\|\nsum{j}c_j\phi_j\right\|_{\mathcal{H}}\leq R_b\left(\nsum{j}|c_j|^2\right)^\frac{1}{2},
\end{equation}
for every sequence $(c_j)\in\ell_2$.  
Following \cite{gardnerbulletin}, we say that $Z$ is a {\em symmetric convex body} if it is a compact, convex set which is symmetric about the origin, and thus contains a ball of some positive radius centered at the origin (we have added the hypothesis here that the body is symmetric about the origin rather than about an arbitrary point).  Given such a body $Z$, we will consider interpolation at a given sequence of points, $X:=(x_j)_{j\in\N}$ in $\R^d$, that is a {\em Riesz-basis sequence} for $L_2(Z)$, namely $(e^{-i\bracket{x_j,\cdot}})_{j\in\N}$ forms a Riesz basis for $L_2(Z)$. In the literature, such sequences are also called {\em complete interpolating sequences} for the associated Paley-Wiener space $PW_Z$.  

It is noted in \cite{bss} that a necessary condition for a sequence $X$ to be a Riesz-basis sequence for $L_2(Z)$ is that it be separated, i.e. there exists a $q>0$ such that
\begin{equation}\label{separationcondition}
 q\leq \|x_k-x_j\|,\quad\textnormal{for all } k\neq j.
\end{equation}

Now we define two operators that will play an important role in our analysis.  First, let $(e_j^\ast)_{j\in\N}\subset L_2(Z)$ be the dual basis of $(e^{-i\bracket{x_j,\cdot}})_{j\in\N}$ which satisfies $\bracket{e^{-i\bracket{x_j,\cdot}},e_k^\ast}=\delta_{j,k}$.  One can show that $(e_j^\ast)$ is also a Riesz basis for $L_2(Z)$ with the same basis constant.  Thus for every $g\in L_2(Z)$, we can write
\begin{equation}\label{EQcoordinaterepresentation}
 g = \nsum{j}\bracket{g,e_j^\ast}_Ze^{-i\bracket{x_j,\cdot}} = \nsum{j}\bracket{g,e^{-i\bracket{x_j,\cdot}}}_Z e_j^\ast,
\end{equation}
where $\bracket{\cdot,\cdot}_Z$ is the usual inner product on $L_2(Z)$.  Putting $g=\F[f]$, the final expression in \eqref{EQcoordinaterepresentation} combined with \eqref{EQRieszbasisconstant} implies that if $f\in PW_Z$, then 
\begin{equation}\label{EQdataell2}
\|f(x_j)\|_{\ell_2}\leq R_b\|\F[f]\|_{L_2(Z)}. 
\end{equation}
Additionally, for any $a\in\R^d$, we have
\begin{equation}\label{EQshiftupperbound}
\left\|\dsum_{j\in\N}\bracket{g,e_j^\ast}_Ze^{-i\bracket{x_j,\cdot}}\right\|_{L_2(a+Z)} \leq R_b^2\|g\|_{L_2(Z)},\end{equation}
where $R_b$ is the basis constant satisfying \eqref{EQRieszbasisconstant}.  Indeed, the left-hand side of \eqref{EQshiftupperbound} is the $L_2(Z)$ norm of $\underset{j\in\N}\sum\bracket{g,e_j^\ast}_Ze^{-i\bracket{a,x_j}}e^{-i\bracket{x_j,\cdot}}$; hence applying \eqref{EQRieszbasisconstant} twice gives an upper bound of
$$R_b\left(\dsum_{j\in\N}\left|\bracket{g,e_j^\ast}_Ze^{-i\bracket{a,x_j}}\right|^2\right)^\frac{1}{2} = R_b\left(\dsum_{j\in\N}\left|\bracket{g,e_j^\ast}_Z\right|^2\right)^\frac{1}{2}\leq R_b^2\left\|\dsum_{j\in\N}\bracket{g,e_j^\ast}_Ze^{-i\bracket{x_j,\cdot}}\right\|_{L_2(Z)},$$
and by \eqref{EQcoordinaterepresentation}, the final term is $R_b^2\|g\|_{L_2(Z)}$.

Consequently, the following extension of $g$ is locally square integrable and thus defined almost everywhere on $\R^d$.
\begin{equation}\label{EQextensiondef}
 E(g)(x):=\nsum{j}\bracket{g,e_j^\ast}_Ze^{-i\bracket{x_j,x}},\quad x\in\R^d.
 \end{equation}

If $m\in\N$, then we define the {\em prolongation operator} $A_m:L_2(Z)\to L_2(Z)$ via
\begin{equation}\label{EQprolongationdef}
 A_m(g)(\xi):=E(g)(2^m\xi)\chi_{Z\setminus\frac{1}{2}Z}(\xi),\quad \xi\in Z,
\end{equation}
where $\chi_S$ is the function taking value 1 on the set $S$ and 0 elsewhere.

It follows from \eqref{EQshiftupperbound} that for $g\in L_2(Z)$,

\begin{align*}
\|A_m(g)\|_{L_2(Z)}^2 = \dint_{Z\setminus\frac{1}{2}Z}|E(g)(2^mu)|^2du & = 2^{-dm}\dint_{2^mZ\setminus2^{m-1}Z}|E(g)(v)|^2dv\\
& \leq2^{-dm}\mathcal{N}R_b^4\|g\|_{L_2(Z)}^2,\notag
\end{align*}
where $\mathcal{N}=\mathcal{N}(Z,\frac{1}{2}Z)$ is the minimum number of translates of $\frac{1}{2}Z$ which are needed to cover $Z$ ($\mathcal{N}$ is typically called the {\em covering number}).  One finds from \cite[Corollary 4.1.15]{geometry} that for any symmetric convex body $Z\subset\R^d$,
$$ 2^d\leq \mathcal{N}\left(Z,\frac{1}{2}Z\right)\leq 5^d.$$  Moreover, $\mathcal{N}(tZ,\frac{t}{2}Z)=\mathcal{N}(Z,\frac{1}{2}Z)$ for any positive $t$ (\cite[Fact 4.1.5]{geometry}). Consequently, 
\begin{equation}\label{EQcovering}
2^d\leq\mathcal{N}\left(2^mZ,2^{m-1}Z\right)\leq 5^d,\quad m\in\N,
\end{equation}
and thus
\begin{equation}\label{EQAmbound}
\|A_m(g)\|_{L_2(Z)}^2\leq 5^d2^{-dm}R_b^4\|g\|_{L_2(Z)}^2.
\end{equation}

As is customary, we use $C$ to denote a constant which may change from line to line, and subscripts may be used to denote dependence on a given parameter.


\section{Interpolation Scheme}\label{SECinterp}

Suppose that $Z\subset\R^d$ is a fixed symmetric convex body.  Also assume that $X:=(x_j)_{j\in\N}$ is a fixed but arbitrary Riesz-basis sequence for $L_2(Z)$ with basis constant $R_b$.  We explore conditions on interpolation operators formed from translates of a single function that allow for recovery of bandlimited functions through a certain limiting process.  The criteria here are inspired by so-called {\em regular interpolators} developed by Ledford \cite{ledford_scattered}.  The results therein are univariate by nature, and our analysis extends to sufficient conditions for interpolation schemes in higher dimensions.

\begin{definition}\label{DEFinterpolator}
 We call a function $\phi:\R^d\to\R$ a $d$-{\em dimensional interpolator for} $PW_Z$ if the following conditions hold.
 
 (I1) $\phi\in L_1(\R^d)\cap C(\R^d)$ and $\widehat{\phi}\in L_1(\R^d)$.
 
 (I2) $\widehat{\phi}\geq0$ and there exists an $\varepsilon>0$ such that $\widehat{\phi}\geq\varepsilon>0$ on $Z$.
 
 (I3) Let $M_j:=\underset{u\in Z\setminus\frac{1}{2}Z}\sup|\widehat{\phi}(2^ju)|$, $j\in\N$.  Then $(2^\frac{jd}{2}M_j)\in\ell_1$.
\end{definition}

It is important to note that for (I2), it is allowable for $\widehat{\phi}$ to be negative everywhere and bounded away from 0 on $Z$, in which case $-\phi$ satisfies the condition.  Condition (I1) allows the use of the Fourier inversion formula \eqref{EQfourierinversion}, while (I2) allows one to show existence of an interpolant for a bandlimited function.  Finally, (I3) is a condition governing how rapidly $\widehat{\phi}$ must decay, and comes from a periodization argument that is ubiquitous throughout the proofs in the sequel. 

\begin{remark}
Condition (I1), which is mainly needed to show that an interpolant exists, can also be stated as follows:
 
 {\em (I1')} $\phi(x) = \dpifrac\dint_{\R^d}\psi(\xi)e^{i\bracket{x,\xi}}d\xi = \F^{-1}[\psi](x)$ for some $\psi\in L_1\cap L_2.$
\end{remark}

\begin{theorem}\label{THMinterpolantbasic}
Let $Z\subset\R^d$ a symmetric convex body.  Suppose that $X$ is a Riesz-basis sequence for $L_2(Z)$, and that $\phi$ is a $d$-dimensional interpolator for $PW_Z$.
 
 (i) For every $f\in PW_Z$, there exists a unique sequence $(a_j)\in\ell_2$ such that
 $$\nsum{j}a_j\phi(x_k-x_j) = f(x_k),\quad k\in\N.$$

 (ii) The Interpolation Operator $\I_\phi:PW_Z\to L_2(\R^d)$ defined by
 $$\I_\phi f(\cdot) = \nsum{j}a_j\phi(\cdot-x_j),$$
 where $(a_j)$ is as in (i), is a well-defined, bounded linear operator from $PW_Z$ to $L_2(\R^d)$.  Moreover, $\I_\phi f$ belongs to $C_0(\R^d)$.  
\end{theorem}

Now we turn to sufficient regularity conditions on a family of $d$-dimensional interpolators to provide convergence to bandlimited functions both in the $L_2$ and uniform norms.  Our terminology is inspired by that of \cite{ledford_scattered}.  Assume that $\delta B_2\subset Z\subset B_2$, where $B_2$ is the Euclidean ball in $\R^d$. 

\begin{definition}\label{DEFregularinterpolator}
 Let $\beta>0$.  Suppose $A\subset(0,\infty)$ is unbounded, and $(\phi_\alpha)_{\alpha\in A}$ is a family of $d$-dimensional interpolators for $PW_Z$.  We call this family {\em regular for} $PW_{\beta B_2}$ if the following hold:
 
 (R1) Let $S_\alpha:=\nsum{j}2^\frac{jd}{2}M_j(\alpha)$ where  $M_j(\alpha)$ is as in (I3) with $\phi$ replaced by $\phi_\alpha$.  For sufficiently large $\alpha$, there is a constant $C$, independent of $\alpha$, such that $S_\alpha\leq CM_\alpha$, where $M_\alpha:=\underset{u\in B_2\setminus\delta B_2}\sup|\widehat{\phi_\alpha}(u)|$.
 
 (R2) Let $m_\alpha(\beta):=\underset{u\in\beta B_2}\inf|\widehat{\phi_\alpha}(u)|$, and $\gamma_\alpha:=\underset{u\in B_2}\inf|\widehat{\phi_\alpha}(u)|$.  Then
 $$\dfrac{M_\alpha^3}{m_\alpha(\beta)\gamma_\alpha^2}\to0, \textnormal{ as } \alpha\to\infty.$$
 
\end{definition}

\begin{remark}\label{REMR3}
 All of the examples considered in Section \ref{SECexamples} are radial basis functions whose Fourier transforms decrease radially.  For such functions, (R2) may be restated as follows:
 
 \textit{(R2')} $\quad\dfrac{\phiahat(\delta)^3}{\phiahat(\beta)\phiahat(1)^2}\to0,\quad\textnormal{as}\quad\alpha\to\infty.$
\end{remark}

We consider interpolation of bandlimited functions $f\in PW_{\beta B_2}$ for some $\beta<\delta$.  Hence, $\F[f]$ has support in a subset of the convex body $Z$.  The condition (R2) comes from exploiting the geometry of the problem, namely that $\beta B_2\subset\delta B_2\subset Z\subset B_2$, and essentially governs how rapidly $\phiahat$ must decay between the balls of radius $\beta$, $\delta$, and 1.  Typically, (R2) will imply a relationship between $\beta$ and $\delta$ which allows the use of the given family of kernels.   We now state our main result.

\begin{theorem}\label{THMmaintheorem}
  Let $d\in\N$, $\delta\in(0,1)$, and $\beta<\delta$.  Suppose that $Z\subset\R^d$ is a symmetric convex body such that $\delta B_2\subset Z\subset B_2$, and suppose $X$ is a Riesz-basis sequence for $L_2(Z)$.  Suppose that $(\phi_\alpha)_{\alpha\in A}$ is a family of $d$-dimensional interpolators for $PW_Z$ that is regular for $PW_{\beta B_2}$, and $\I_\alpha$ are the associated interpolation operators.
Then for every $f\in PW_{\beta B_2}$, 
 $$\inflim{\alpha}\|\I_\alpha f-f\|_{L_2(\R^d)} = 0,$$
 and
 $$\inflim{\alpha}|\I_\alpha f(x)-f(x)| = 0, \quad\textnormal{uniformly on }\R^d.$$
 Moreover, 
 $$\max\left\{\|\I_\alpha f-f\|_{L_2(\R^d)}, \;\|\I_\alpha f-f\|_{L_\infty(\R^d)}\right\}\leq C\dfrac{M_\alpha^3}{m_\alpha(\beta)\gamma_\alpha^2}\|f\|_{L_2(\R^d)}.$$
\end{theorem}

\section{Examples}\label{SECexamples}

Since the conditions given above are somewhat abstract, it is prudent to pause and discuss some examples that motivate the general result. Throughout this section, suppose that $Z$ and $X$ satisfy the hypothesis of Theorem \ref{THMmaintheorem}.  We begin with a pair of lemmas which will be useful in checking the regularity conditions for the subsequent examples.

\begin{lemma}\label{LEMexpformula}
Let $D$ and $a$ be positive.  Then for every $n\in\N$, the following holds:
\begin{multline*}\dint_1^\infty D^x e^{-a2^{x-1}}dx = D\left[\frac{1}{a\ln2}+\finsum{j}{2}{n}\frac{1}{(a\ln2)^j}\product{k}{1}{j-1}\ln\left(\frac{D}{2^k}\right)\right]e^{-a}\\+\left(\dfrac{2}{a\ln2}\right)^{n}\;\product{k}{1}{n}\ln\left(\dfrac{D}{2^k}\right)\dint_1^\infty\left(\frac{D}{2^{n}}\right)^xe^{-a2^{x-1}}dx.
\end{multline*}
Additionally, $$\dint_1^\infty2^xe^{-a2^{x-1}}dx = \dfrac{2}{a\ln2}e^{-a}.$$
\end{lemma}

\begin{proof}
Note that for the base case,
$$\dint_1^\infty D^xe^{-a2^{x-1}}dx = \dint_1^\infty\left(\frac{D}{2}\right)^x2^xe^{-a2^{x-1}}dx = \dfrac{D}{a\ln2}e^{-a}+\dfrac{2}{a\ln2}\ln\left(\dfrac{D}{2}\right)\dint_1^\infty \left(\dfrac{D}{2}\right)^xe^{-a2^{x-1}}dx,$$
which follows from integrating by parts and noticing that $\inflim{x}(\frac{D}{2})^xe^{-a2^{x-1}}=0$.  The desired formula is then obtained via induction on $n$.  

For the second statement, one need only check that $-\frac{2}{a\ln 2}e^{-a2^{x-1}}$ is a primitive of the integrand.
\end{proof}

\begin{lemma}\label{LEMseries}
Let $D>1$.  If $a\geq\max\{\frac{\ln D}{\ln 2},\frac{2}{\ln2}\}$, then  $$\nsum{j}D^je^{-a2^{j-1}}\leq C_De^{-a}.$$
\end{lemma}

\begin{proof}
The condition on the size of $a$ implies that the function $f(x)=D^xe^{-a2^{x-1}}$ is decreasing on $[1,\infty)$, and consequently, the series in question is majorized by $$De^{-a}+\dint_1^\infty D^xe^{-a2^{x-1}}dx.$$  If $D\leq2$, then estimate the integral above by $\frac{2}{a\ln 2}e^{-a}$ (which is at most $e^{-a}$) by the second statement of Lemma \ref{LEMexpformula}.  Otherwise, apply the conclusion of Lemma \ref{LEMexpformula} with $n$ such that $2^{-n}D\leq2$, and notice that 
$$\dint_1^\infty D^xe^{-a2^{x-1}}dx\leq D\left[n\ln^n\left(\frac{D}{2}\right)\right]e^{-a}+\ln^n\left(\frac{D}{2}\right)\left|\ln\left(\frac{D}{2^n}\right)\right|\dint_1^\infty 2^xe^{-a2^{x-1}}dx.$$
The right hand side above is bounded by a constant depending on $D$ and $n$ (which depends on $D$) times $e^{-a}$, using the fact that $\frac{2}{a\ln2}\leq1$, and the conclusion follows.
\end{proof}

We now begin our examples with the Gaussian kernel and show that we recover the main result from \cite{bss}.

\subsection{Gaussians}

To fit the imposed condition of $\alpha$ tending to infinity, we use a different convention for the Gaussian kernel than \cite{bss}:
$$g_\alpha(x):= e^{-\frac{\|x\|^2}{4\alpha}},\quad \alpha\geq1,\quad x\in\R^d.$$ 
Thus $\widehat{g_\alpha}(\xi)=(2\alpha)^{-\frac{d}{2}}e^{-\alpha\|\xi\|^2}.$  Conditions (I1)-(I3) are readily verified, and will be discussed in a subsequent example. Evidently, $\widehat{g_\alpha}$ is radially decreasing, so $M_\alpha = (2\alpha)^{-\frac{d}{2}}e^{-\alpha\delta^2}$ and $M_j(\alpha) \leq (2\alpha)^{-\frac{d}{2}}e^{-\alpha2^{2(j-1)}\delta^2}$.  Therefore, to check condition (R1), note that for sufficiently large $\alpha$ (specifically, $\alpha\delta^2\geq d/2$), Lemma \ref{LEMseries} implies that
$$S_\alpha \leq (2\alpha)^{-\frac{d}{2}}\nsum{j}2^\frac{jd}{2}e^{-\alpha2^{j-1}\delta^2}\leq C(2\alpha)^{-\frac{d}{2}}e^{-\alpha\delta^2} = CM_\alpha,$$
where $C$ is some constant depending only on the dimension $d$. Note that to apply Lemma \ref{LEMseries}, we made use of the fact that $2^{2(j-1)}\alpha\delta^2\geq2^{j-1}\alpha\delta^2$.  Considering (R2') and noting that $m_\alpha(\beta) = \widehat{g_\alpha}(\beta)$, and $\gamma_\alpha = \widehat{g_\alpha}(1)$, we find that  
$$\dfrac{M_\alpha^3}{m_\alpha(\beta)\gamma_\alpha^2} \leq e^{\alpha(\beta^2+2-3\delta^2)},$$
and the latter tends to 0 as $\alpha\to\infty$ provided $\beta<\sqrt{3\delta^2-2}$.  This, in turn, requires $\delta>\sqrt{2/3}$ since $\beta$ must be positive. Consequently, the result of Theorem \ref{THMmaintheorem} coincides with the main theorem in \cite{bss}, which we reproduce here in our terminology.

\begin{theorem}[cf. \cite{bss}, Theorem 3.6]\label{THMGaussian}
 Let $\delta\in(\sqrt{2/3},1)$ and $\beta\in(0,\sqrt{3\delta^2-2})$.  Then the set of Gaussians $\left(e^{-\frac{\|\cdot\|^2}{4\alpha}}\right)_{\alpha\in[1,\infty)}$ is a family of $d$-dimensional interpolators for $PW_Z$ that is regular for $PW_{\beta B_2}$.  In particular, for every $f\in PW_{\beta B_2}$, we have $\inflim{\alpha}\I_\alpha f = f$ in $L_2(\R^d)$ and uniformly on $\R^d$.  Moreover,
 $$\max\left\{\|\I_\alpha f-f\|_{L_2(\R^d)},\;\|\I_\alpha f-f\|_{L_\infty(\R^d)}\right\}\leq Ce^{\alpha(\beta^2+2-3\delta^2)}\|f\|_{L_2(\R^d)}.$$
\end{theorem}

\subsection{Inverse Multiquadrics}

Our next example is a family of {\em inverse multiquadrics}.  For an exponent $\nu>d/2$, we define the general inverse multiquadric with shape parameter $c>0$ by
$$\phi_{\nu,c}(x):=\dfrac{1}{(\|x\|^2+c^2)^\nu},\quad x\in\R^d.$$
We will consider the regularity of the family $(\phi_{\nu,c})_{c\in[1,\infty)}$.  This example requires a bit more work up front since the Fourier transform of the inverse multiquadric is somewhat complicated.  For now, suppose $\nu$ is fixed, and we suppress the dependence on $\nu$ and write $\phi_c$ for notational ease.  

Since $\nu>d/2$, $\phi_c$ is integrable for all $c>0$, and from \cite[Theorem 8.15]{Wendland}, we find that its Fourier transform is given by the equation
\begin{equation}\label{EQphiFT}
 \widehat{\phi_c}(\xi) = \dfrac{2^{1-\nu}}{\Gamma(\nu)}\left(\dfrac{\|\xi\|}{c}\right)^{\nu-\frac{d}{2}}K_{\nu-\frac{d}{2}}(c\|\xi\|),\quad \xi\neq0,
\end{equation}
where $\Gamma$ is the usual gamma function, and $K_\gamma$ is the univariate modified Bessel function of the second kind defined as follows.
\begin{equation}\label{EQKdef}
 K_\gamma(r) := \dint_0^\infty e^{-r\cosh t}\cosh(\gamma t)dt,\quad \gamma\in\R,\quad r>0.
\end{equation}
It is important to note from the definition that $K_\gamma$ is symmetric with respect to its order. That is, $K_{-\gamma} = K_\gamma$.  Additionally, $K_\gamma(r)>0$ for all $r>0$ and all $\gamma\in\R$.

We begin with some necessary properties of both the inverse multiquadrics and the modified Bessel functions of the second kind.  First, we have the following differentiation formula (\cite[p. 361]{AandS}):
\begin{equation}\label{EQbesseldiff}
 \dfrac{d}{dr}\left[r^\gamma K_\gamma(r)\right] = -r^\gamma K_{\gamma-1}(r).
\end{equation}

This leads to the following observation.

\begin{proposition}\label{PROPphidecreasing}
 The function $\widehat{\phi_c}$ is always positive, and is radially decreasing. That is, if $\|x\|\leq\|y\|$, then $\widehat{\phi_c}(x)\geq\widehat{\phi_c}(y)$.
\end{proposition}

\begin{proof}
 Positivity is evident from \eqref{EQphiFT} and the fact that $K_\gamma(r)>0$.  To see that $\widehat{\phi_c}$ is decreasing, set $r=\|x\|$, and note that \eqref{EQphiFT} and \eqref{EQbesseldiff} imply
 $$\dfrac{d}{dr}\left[\widehat{\phi_c}(r)\right] = \dfrac{2^{1-\nu}}{\Gamma(\nu)c^{\nu-\frac{d}{2}}}\dfrac{d}{dr}\left[r^{\nu-\frac{d}{2}}K_{\nu-\frac{d}{2}}(cr)\right] = -\dfrac{2^{1-\nu}}{\Gamma(\nu)c^{\nu-\frac{d}{2}-1}}r^{\nu-\frac{d}{2}}K_{\nu-\frac{d}{2}-1}(cr)<0.$$
\end{proof}

The following inequalities are summarized from \cite[Section 5.1]{Wendland}.

\begin{proposition}\label{PROPKinequalities}
 (i) If $\nu-d/2\geq 1/2$, then $$K_{\nu-\frac{d}{2}}(r)\geq\sqrt{\dfrac{\pi}{2}}\,r^{-1/2}e^{-r},\quad r>0.$$
 
 (ii) If $\nu-d/2<1/2$ and $r>1$, then $$K_{\nu-\frac{d}{2}}(r)\geq C_{\nu,d}\,r^{-1/2}e^{-r}, \quad\textnormal{ where}\quad C_{\nu,d}:=\dfrac{\sqrt{\pi}\,3^{\nu-\frac{d}{2}-\frac{1}{2}}}{2^{\nu-\frac{d}{2}+1}\Gamma\left(\nu-\frac{d}{2}+\frac{1}{2}\right)}.$$
 
 (iii) $$K_{\nu-\frac{d}{2}}(r)\leq\sqrt{2\pi}\,r^{-1/2}e^{-r}e^\frac{|\nu-\frac{d}{2}|^2}{2r},\quad r>0.$$
 
 (iv) $$K_{\nu-\frac{d}{2}}(r)\leq 2^{\nu-\frac{d}{2}-1}\Gamma\left(\nu-\frac{d}{2}\right)r^{\frac{d}{2}-\nu},\quad r>0.$$
\end{proposition}

To show (I1), it is evident from the definition that $\phi_c$ is integrable, and the following proposition shows that $\widehat{\phi_c}$ is as well.

\begin{proposition}\label{PROPphiintegrable}
 For $\nu>d/2$ and $c>0$, $\widehat{\phi_c}\in L_1(\R^d)$.
\end{proposition}

\begin{proof}
According to \eqref{EQphiFT}, we need only show that $\int_{\R^d}\|\xi\|^{\nu-\frac{d}{2}}K_{\nu-\frac{d}{2}}(c\|\xi\|)d\xi$ converges.  We split this into two pieces, the integral over the Euclidean ball and the integral outside.  By Proposition \ref{PROPKinequalities} {\em (iv)}, we see that
$$I_1:=\dint_{B_2}\|\xi\|^{\nu-\frac{d}{2}}K_{\nu-\frac{d}{2}}(c\|\xi\|)d\xi\leq C\dint_{B_2}\|\xi\|^{\nu-\frac{d}{2}}\|\xi\|^{\frac{d}{2}-\nu}d\xi = C\;m(B_2),$$
where $C$ is a finite constant depending on $\nu, d$, and $c$, and $m(B_2)$ is the Lebesgue measure of the Euclidean ball.
Furthermore, by Proposition \ref{PROPKinequalities}{\em(iii)}, 
\begin{displaymath}
\begin{array}{lllll}
I_2 & := & \dint_{\R^d\setminus B_2}\|\xi\|^{\nu-\frac{d}{2}}K_{\nu-\frac{d}{2}}(c\|\xi\|)d\xi & \leq & C\dint_{\R^d\setminus B_2}\|\xi\|^{\nu-\frac{d}{2}-\frac{1}{2}}e^{-c\|\xi\|}e^\frac{|\nu-\frac{d}{2}|^2}{2c\|\xi\|}d\xi\\
\\
& & &\leq & C\dint_{\R^d\setminus B_2}\|\xi\|^{\nu-\frac{d}{2}-\frac{1}{2}}e^{-c\|\xi\|}d\xi,\\
\end{array}
\end{displaymath}
and the right hand side is a convergent integral.  Again, $C$ is a finite constant depending on $\nu, d$, and $c$.  In the final inequality, we have used the fact that $e^\frac{|\nu-\frac{d}{2}|^2}{2c\|\xi\|}\leq e^\frac{|\nu-\frac{d}{2}|^2}{2c}$.
\end{proof}

Next notice that (I2) follows from Proposition \ref{PROPphidecreasing} and the fact that $\widehat{\phi_c}(1)>0.$  Thus it remains to check (I3) and the regularity conditions.  By Propositions \ref{PROPphidecreasing} and \ref{PROPKinequalities}(iii), 
\begin{equation}\label{EQMjmultiquadric}
M_j(c)\leq |\widehat{\phi_c}(2^{j-1}\delta)|\leq C_\nu \left(\dfrac{2^{j-1}\delta}{c}\right)^{\nu-\frac{d}{2}}(c\delta)^{-\frac{1}{2}}e^{-c2^{j-1}\delta}e^\frac{|\nu-\frac{d}{2}|^2}{2^jc\delta}. 
\end{equation}
The right hand side of \eqref{EQMjmultiquadric} is summable for any fixed $c$, which yields (I3).

To check (R1), we may assume, without loss of generality, that $c$ is large enough so that the final exponential term on the right hand side of \eqref{EQMjmultiquadric} is at most 2, in which case we have (by Proposition \ref{PROPKinequalities}(ii) and Lemma \ref{LEMseries}) that for sufficiently large $c$,
$$\nsum{j}2^\frac{jd}{2}M_j(c)\leq C_\nu\nsum{j}2^\frac{jd}{2}2^{(j-1)(\nu-\frac{d}{2})}\left(\frac{\delta}{c}\right)^{\nu-\frac{d+1}{2}}e^{-c2^{j-1}\delta}\leq C_{\nu,d}\left(\dfrac{\delta}{c}\right)^{\nu-\frac{d+1}{2}}e^{-c\delta}\leq C_{\nu,d}\widehat{\phi_c}(\delta).$$

Finally, we check (R2').  By Proposition \ref{PROPKinequalities}, we find that 
$$\frac{\widehat{\phi_c}(\delta)^3}{\widehat{\phi_c}(\beta)\widehat{\phi_c}(1)^2} \leq C_{\nu,d}\left(\dfrac{\delta}{\beta}\right)^{\nu-\frac{d+1}{2}}c^{\frac{d-1}{2}-\nu}e^{c(\beta+2-3\delta)}.$$
Consequently, as long as $0<\beta<3\delta-2$ and $\delta>2/3$, (R2') is satisfied.  We summarize this in the following theorem.

\begin{theorem}\label{THMmultiquadric}
 Let $\nu>d/2$.  Assume $\delta\in(2/3,1)$ and $\beta\in(0,3\delta-2)$.  Then the set of inverse multiquadrics $\left((\|x\|^2+c^2)^{-\nu}\right)_{c\in[1,\infty)}$ is a family of $d$-dimensional interpolators for $PW_Z$ that is regular for $PW_{\beta B_2}$.  In particular, for every $f\in PW_{\beta B_2}$, we have $\inflim{c}\I_c f = f$ in $L_2(\R^d)$ and uniformly on $\R^d$.  Moreover,
 $$\max\left\{\|\I_c f-f\|_{L_2(\R^d)},\; \|\I_cf-f\|_{L_\infty(\R^d)}\right\}\leq C\left(\dfrac{\delta}{c\beta}\right)^{\nu-\frac{d+1}{2}}e^{c(\beta+2-3\delta)}.$$
\end{theorem}

\subsection{A Broad Class of Examples}\label{SECclassexamples}

We end with a large class of examples which includes both the Gaussian and the Poisson kernel as specific cases.  These classes provide natural extensions of the results in \cite{bss}.  For any $p>0$, we define the following function with parameter $\alpha$:
\begin{equation}\label{EQpstabledef}
 g_\alpha(x):=\dpifrac\dint_{\R^d}e^{-\alpha\|\xi\|^p}e^{i\bracket{x,\xi}}d\xi,\quad x\in\R^d,
\end{equation}
or in other words, $g_\alpha = \F^{-1}\left[e^{-\alpha\|\cdot\|^p}\right]$.  We note that in the case $d=1$ and $p\leq2$, these classes correspond to the so-called $p$-stable random variables. 

By definition, $g_\alpha$ satisfies (I1').  Condition (I2) is evident, and to check (I3), note that since $\widehat{g_\alpha}$ is radially decreasing,
$M_j(\alpha) \leq e^{-\alpha2^{(j-1)p}\delta^p}$, and thus $(2^\frac{jd}{2}M_j(\alpha))_{j\in\N}$ is summable.  We now check the regularity conditions, which will give us bounds on $\beta$ and $\delta$ as in the previous examples.  Note that $M_\alpha = e^{-\alpha\delta^p}$.  Then 
$$S_\alpha \leq \nsum{j}2^\frac{jd}{2}e^{-\alpha2^{(j-1)p}\delta^p},$$
which for $p\geq1$ is at most $Ce^{-\alpha\delta^p}=CM_\alpha$ by Lemma \ref{LEMseries}.  For $p<1$, we cannot simply apply the conclusion of Lemma \ref{LEMseries}; however, a straightforward adaptation of the proof there utilizing the fact that $\int_1^\infty 2^{px}e^{-a2^{p(x-1)}}dx=\frac{2^p}{a\ln2}$ implies that for sufficiently large $\alpha$, we may bound $S_\alpha$ by $Ce^{-\alpha\delta^p}$, where $C$ is a constant independent of $\alpha$.  

Per Remark \ref{REMR3}, we consider (R2') as follows:
$$\dfrac{M_\alpha}{m_\alpha(\beta) \gamma_\alpha^2}=\dfrac{g_\alpha(\delta)^3}{g_\alpha(\beta)g_\alpha(1)^2} = e^{\alpha(\beta^p+2-3\delta^p)}.$$
Evidently, the right hand side tends to 0 as $\alpha\to\infty$ whenever $\beta<(3\delta^p-2)^\frac{1}{p}$.  We conclude the following.

\begin{theorem}\label{THMgeneralexample}
 Let $p>0$.  Suppose $\delta\in\left(\left(\frac{2}{3}\right)^\frac{1}{p},1\right)$, and $0<\beta<(3\delta^p-2)^\frac{1}{p}$.  Then  $(g_\alpha)_{\alpha\in(0,\infty)}$ defined by \eqref{EQpstabledef} is a family of $d$-dimensional interpolators for $PW_Z$ that is regular for $PW_{\beta B_2}$.  In particular, for every $f\in PW_{\beta B_2}$, we have $\inflim{\alpha}\I_\alpha f = f$ in $L_2(\R^d)$ and uniformly on $\R^d$.  Moreover,
 $$\max\left\{\|\I_\alpha f-f\|_{L_2(\R^d)},\;\|\I_\alpha f-f\|_{L_\infty(\R^d)}\right\}\leq Ce^{\alpha(\beta^p+2-3\delta^p)}\|f\|_{L_2(\R^d)}.$$	
\end{theorem}

Note that in the case $p=2$, $g_\alpha$ is the Gaussian discussed in the first example, and the condition reads $0<\beta<\sqrt{3\delta^2-2}$; so in this case, Theorems  \ref{THMGaussian} and \ref{THMgeneralexample} coincide.


\section{Proofs}\label{SECproofs}

\subsection{Proof of Theorem \ref{THMinterpolantbasic}}

Throughout this section, assume that $Z$ is as in the statement of the theorem and that $X$ is a Riesz-basis sequence for $L_2(Z)$.  Recall that $A_j$ is the prolongation operator defined by \eqref{EQprolongationdef}.  Our first step in the proof is the following key lemma.  

\begin{lemma}\label{LEMmatrixbounded}
 Suppose that $\phi$ is a $d$-dimensional interpolator for $PW_Z$, and let $A:=(\phi(x_m-x_n))_{m,n\in\N}$.  Then $A:\ell_2\to\ell_2$ is a bounded, invertible, linear operator.
\end{lemma}

\begin{proof}
 Linearity is plain, so we will take up boundedness first by looking at $\bracket{Aa,a}_{\ell_2}$ for arbitrary $a:=(a_j)_{j\in\N}\in\ell_2$.  To show boundedness, we use the Dominated Convergence Theorem, (I1), and a periodization argument to see that 
 
 \begin{displaymath}
  \begin{array}{lll}
   \nsum{m,n}a_m\overline{a_n}\phi(x_m-x_n) & = & \nsum{m,n}a_m\overline{a_n}\dpifrac\dint_{\R^d}\phihat(\xi)e^{i\bracket{x_m-x_n,\xi}}d\xi\\
   \\
   & = & \dpifrac\dint_{\R^d} \phihat(\xi)\left|\nsum{n}a_ne^{i\bracket{x_n,\xi}}\right|^2d\xi\\
   \\
   & = & \dpifrac\left[\dint_Z\phihat(\xi)\left|\nsum{n}a_ne^{i\bracket{x_n,\xi}}\right|^2d\xi + \nsum{j}\dint_{2^jZ\setminus2^{j-1}Z}\phihat(\xi)\left|\nsum{n}a_ne^{i\bracket{x_n,\xi}}\right|^2d\xi\right].\\
   \end{array}\end{displaymath}
By \eqref{EQRieszbasisconstant}, the first term above is majorized by $(2\pi)^{-\frac{d}{2}}\underset{u\in Z}\sup|\widehat{\phi}(u)|R_b^2\|a\|_{\ell_2}^2,$ while the second (via a change of variables) by
$$\dpifrac\nsum{j}2^{jd}\dint_{Z\setminus\frac{1}{2}Z}\phihat(2^j\xi)   \left|A_j\left(\nsum{n}a_ne^{i\bracket{x_n,\xi}}\right)\right|^2d\xi\leq \dpifrac\nsum{j}2^{jd}M_j\left\|A_j\left(\nsum{n}a_ne^{-i\bracket{x_n,\cdot}}\right)\right\|_{L_2(Z)}^2.$$  Application of \eqref{EQAmbound} and \eqref{EQRieszbasisconstant} gives that the above is bounded by
$(2\pi)^{-\frac{d}{2}}5^dR_b^6\|(M_j)_j\|_{\ell_1}\|a\|_{\ell_2}^2,$ which is bounded on account of (I3).  We could equivalently have used (I1') in the first line as we simply needed to write $\phi(x_m-x_n)$ via its Fourier integral.  

To show invertibility, we will find a lower bound for the inner product. Indeed, using the Dominated Convergence Theorem again along with (I2) and \eqref{EQRieszbasisconstant}, we find that

\begin{displaymath}
 \begin{array}{lll}
   \nsum{m,n}a_m\overline{a_n}\phi(x_m-x_n) & = &\dpifrac\dint_{\R^d} \phihat(\xi)\left|\nsum{n}a_ne^{i\bracket{x_n,\xi}}\right|^2d\xi\\
   \\
   & \geq & \dpifrac\dint_Z\phihat(\xi)\left|\nsum{n}a_ne^{i\bracket{x_n,\xi}}\right|^2d\xi\\
   \\
   & \geq & \dfrac{\varepsilon}{R_b^2(2\pi)^\frac{d}{2}}\|a\|_{\ell_2}^2.\\
 \end{array}
\end{displaymath}

\end{proof}

\begin{proof}[Proof of Theorem \ref{THMinterpolantbasic}]

Note that (i) is a direct consequence of Lemma \ref{LEMmatrixbounded} and \eqref{EQdataell2}.  

To show (ii), we first prove that the function $\omega :=\phihat\;\nsum{n}a_ne^{i\bracket{x_n,\cdot}}$ belongs to $L_1\cap L_2$, which we do by the same periodization argument as in the proof of Lemma \ref{LEMmatrixbounded}:
 
 $$\dint_{\R^d}|\phihat(\xi)|\left|\nsum{n}a_ne^{i\bracket{x_n,\xi}}\right|d\xi  \leq  \underset{u\in Z}\sup|\phihat(u)|\left\|\nsum{n}a_je^{i\bracket{x_n,\cdot}}\right\|_{L_1(Z)} + \nsum{j}2^{jd}M_j\left\|A_j\left(\nsum{n}a_ne^{i\bracket{x_n,\cdot}}\right)\right\|_{L_1(Z)}.$$
By the Cauchy-Schwarz inequality, the first term is bounded by $\underset{u\in Z}\sup|\phihat(u)|m(Z)^\frac{1}{2}R_b\|a\|_{\ell_2}$, whilst the second by $5^dR_b^3m(Z)^\frac{1}{2}\|(2^\frac{jd}{2}M_j)_j\|_{\ell_1}\|a\|_{\ell_2}$ due to \eqref{EQAmbound}, which is bounded on account of (I3). 
 
The argument for square-integrability follows similar reasoning as in the first part of the proof of Lemm \ref{LEMmatrixbounded}:
$$\dint_{\R^d}|\phihat(\xi)|^2\left|\nsum{n}a_ne^{i\bracket{x_n,\xi}}\right|^2d\xi  \leq  \underset{u\in Z}\sup|\phihat(u)|^2R_b^2\|a\|_{\ell_2}^2+5^dR_b^6\nsum{j}M_j^2\|a\|_{\ell_2}^2.$$

Since $\|(M_j)_j\|_{\ell_2}^2\leq\|(M_j)_j\|_{\ell_1}^2$, (I3) implies that the last term in the above inequality is finite.  Consequently, $\omega\in L_1\cap L_2$.  It follows from basic techniques and the Riemann-Lebesgue Lemma that the function
$$\I_\phi f(x) = \dpifrac\dint_{\R^d}\omega(\xi)e^{i\bracket{\xi,x}}d\xi = \nsum{j}a_j\phi(x-x_j)$$
belongs to $C_0(\R^d)\cap L_2(\R^d)$, and moreover that $\F[\I_\phi f] = \omega$.
 
Finally, to conclude boundedness, simply notice from the periodization argument above, Lemma \ref{LEMmatrixbounded}, Plancherel's Identity, and \eqref{EQdataell2}, that $$\|\I_\phi f\|_{L_2} = \|\F[\I_\phi f]\|_{L_2}\leq C\|a\|_{\ell_2} \leq C\|A^{-1}\|_{\ell_2\to\ell_2}\|f(x_k)\|_{\ell_2}\leq C\|A^{-1}\|_{\ell_2\to\ell_2}R_b\|f\|_{L_2}.$$
 
\end{proof}

\subsection{Proof of Theorem \ref{THMmaintheorem}}

We now embark on the proof of our main result. Let $Z, X, \delta$, and $\beta$ be as in the statement of Theorem \ref{THMmaintheorem}, and let $(\phi_\alpha)_{\alpha\in A}$ be a family of $d$-dimensional interpolators for $PW_Z$ that is regular for $PW_{\beta B_2}$.  Let $M_\alpha$ and $\gamma_\alpha$ be as in (R1) and (R2), respectively, and throughout the sequel, assume that $\alpha$ is sufficiently large for (R1) to hold.  The first step is to show that there exists a constant $C<\infty$ so that
$$\|\F[\I_\alpha f]\|_{L_2(Z)}\leq C\dfrac{M_\alpha}{\gamma_\alpha}\|\F[f]\|_{L_2(Z)},$$ for every $f\in PW_Z$.  We proceed in a series of steps following the techniques of \cite{bss}.

To begin, define the function
\begin{equation}\label{EQpsidef}
\Psi_\alpha(u):=\nsum{j}a_je^{-i\bracket{x_j,u}} = \frac{1}{\phiahat(u)}\F[\I_\alpha f](u),\quad u\in\R^d,
\end{equation}
and let $\psi_\alpha$ denote the restriction of $\Psi_\alpha$ to $Z$.
\begin{remark}\label{REMRiesz}
It is important to note that by uniqueness of the Riesz basis representation for a function on $Z$, we have that $\Psi_\alpha(u) = E(\psi_\alpha)(u)$ on $\R^d$.  That is, $\Psi_\alpha$ is defined globally by its Riesz basis representation on the body $Z$.  This fact is crucial to the subsequent analysis.
\end{remark}

\begin{lemma}\label{LEMidentity}
 The following holds:
 $$\F[f] = \F[\I_\alpha f]+ \nsum{m}2^{dm}A_m^\ast\left(\widehat{\phi_\alpha}(2^m\cdot)A_m(\psi_\alpha)\right) \textnormal{ a.e. on } Z.$$ 
\end{lemma}

\begin{proof}
 We note that since $\left(e^{-i\bracket{x_j,\cdot}}\right)$ is a Riesz basis for $L_2(Z)$, it suffices to show that the inner product of both sides above with respect to the basis elements are all equal.  First, by \eqref{EQfourierinversion},
 $$\bracket{\F[f],e^{-i\bracket{x_j,\cdot}}}_Z = (2\pi)^df(x_j).$$
 On the other hand, the interpolation condition guarantees that
 \begin{displaymath}
  \begin{array}{lll}
   (2\pi)^df(x_j) & = & (2\pi)^d\I_\alpha f(x_j)\\
   \\
   & = & \dint_{\R^d}\F[\I_\alpha f](u)e^{i\bracket{x_j,u}}du\\
   \\
   & = & \dint_Z\F[\I_\alpha f](u)e^{i\bracket{x_j,u}}du + \nsum{m}\dint_{2^mZ\setminus2^{m-1}Z}\widehat{\phi_\alpha}(u)\Psi_\alpha(u)e^{i\bracket{x_j,u}}du\\
   \\
   & =: & I_1 + I_2. \\
  \end{array}
 \end{displaymath}
 
 Evidently, $I_1 = \bracket{\F[\I_\alpha f],e^{-i\bracket{x_j,\cdot}}}_Z$.  Now
  \begin{displaymath}
  \begin{array}{lll}
   I_2 & = & \nsum{m}2^{dm}\dint_{Z\setminus\frac{1}{2}Z}\phiahat(2^mv)\Psi_\alpha(2^mv)e^{i\bracket{x_j,2^mv}}dv\\
   \\
   & = & \nsum{m}2^{dm}\dint_{Z\setminus\frac{1}{2}Z}\phiahat(2^mv)A_m(\psi_\alpha)(v)A_m\left(e^{i\bracket{x_j,\cdot}}\right)(v)dv\\
   \\
   & = & \nsum{m}2^{dm}\bracket{\phiahat(2^m\cdot)A_m(\psi_\alpha),A_m\left(e^{-i\bracket{x_j,\cdot}}\right)}_Z\\
   \\
   & = & \nsum{m}2^{dm}\bracket{A_m^\ast\left(\phiahat(2^m\cdot)A_m(\psi_\alpha)\right),e^{-i\bracket{x_j,\cdot}}}_Z,\\
  \end{array}
 \end{displaymath}
whence the identity.
\end{proof}

We now define an operator that is implicit in the previous Lemma:
\begin{equation}\label{EQtaucdef}
 \tau_\alpha:L_2(Z)\to L_2(Z),\quad\textnormal{via}\quad \tau_\alpha(h):=\nsum{m}A_m^\ast\left(\phiahat(2^m\cdot)A_m(h)\right).
 \end{equation}
 
 \begin{proposition}\label{PROPtauc}
The operator $\tau_\alpha$ defined by \eqref{EQtaucdef} is a bounded linear operator on $L_2(Z)$ that is positive, (i.e.  $\bracket{\tau_\alpha(h),h}_Z\geq0$ for all $h\in L_2(Z)$).  Moreover, there exists a positive number $C$, which is independent of $\alpha$, so that
  \begin{equation}\label{EQtaucbound}
   \|\tau_\alpha\|\leq CM_\alpha.
  \end{equation}

 \end{proposition}

 \begin{proof}
  Linearity is plain, and positivity can be seen as follows.
  $$\bracket{\tau_\alpha(h),h}_Z = \nsum{m}2^{dm}\bracket{\phiahat(2^m\cdot)A_m(h),A_m(h)}_Z = \nsum{m}2^{dm}\dint_{Z\setminus\frac{1}{2}Z}\phiahat(2^mu)|A_m(h)(u)|^2du\geq0,$$  the final inequality stemming from the positivity of $\phiahat$.
  
 To prove the upper bound, notice that \eqref{EQAmbound} implies that for $h\in L_2(Z)$,
 \begin{displaymath}
  \begin{array}{lll}
   \|\tau_\alpha(h)\|_{L_2(Z)} & \leq & \nsum{m}2^{dm}\left\|A_m^\ast\left(\phiahat(2^m\cdot)A_m(h)\right)\right\|_{L_2(Z)}\\
   \\
   & \leq & 5^\frac{d}{2}R_b^2\nsum{m}2^\frac{dm}{2}M_m(\alpha)\|A_m(h)\|_{L_2(Z)}\\
   \\
   & \leq & 5^{d}R_b^4\nsum{m}M_m(\alpha)\|h\|_{L_2(Z)}\\
   \\
   & \leq & 5^{d}R_b^4S_\alpha\|h\|_{L_2(Z)}\\
   \\
   & \leq & CM_\alpha\|h\|_{L_2(Z)}.\\
  \end{array}
 \end{displaymath}
 
 The final inequality comes from condition (R1).
 \end{proof}

Next, note that the positivity of $\tau_\alpha$ and Lemma \ref{LEMidentity} imply that
$$\|\F[f]\|_{L_2(Z)}\|\psi_\alpha\|_{L_2(Z)}\geq\bracket{\F[f],\psi_\alpha}_Z\geq\bracket{\F[\I_\alpha f],\psi_\alpha}_Z\geq \gamma_\alpha\|\psi_\alpha\|_{L_2(Z)}^2.$$

Therefore,

\begin{equation}\label{EQpsiupperbound}
 \|\psi_\alpha\|_{L_2(Z)}\leq \dfrac{1}{\gamma_\alpha}\|\F[f]\|_{L_2(Z)}.
\end{equation}

From Lemma \ref{LEMidentity}, Proposition \ref{PROPtauc}, and \eqref{EQpsiupperbound}, we see that
\begin{equation}\label{EQIcnearorigin}
 \|\F[\I_\alpha f]\|_{L_2(Z)}\leq C\dfrac{M_\alpha}{\gamma_\alpha}\|\F[f]\|_{L_2(Z)}.
\end{equation}

Now we estimate $\|\F[\I_\alpha f]\|_{L_2(\R^d\setminus Z)}$.  We accomplish this by a familiar periodization argument and utilization of \eqref{EQAmbound} and \eqref{EQpsiupperbound}:

\begin{displaymath}
 \begin{array}{lll}
  \|\F[\I_\alpha f]\|_{L_2(\R^d\setminus Z)}^2 & = & \nsum{m}\dint_{2^mZ\setminus2^{m-1}Z}|\phiahat(u)|^2|\Psi_\alpha(u)|^2du\\
  \\
  & \leq & \nsum{m}2^{dm}M_m(\alpha)^2\|A_m(\psi_\alpha)\|_{L_2(Z)}^2\\
  \\
  & \leq & 5^dR_b^4\nsum{m}M_m(\alpha)^2\|\psi_\alpha\|_{L_2(Z)}^2\\
  \\
  & \leq & 5^dR_b^4\dfrac{1}{\gamma_\alpha^2}\|\F[f]\|_{L_2(Z)}^2\nsum{m}M_m(\alpha)^2.\\
  
 \end{array}
\end{displaymath}
It remains to note that the series in the final expression above is majorized by
$$ \left(\nsum{m}M_m(\alpha)\right)^2 \leq S_\alpha^2\leq CM_\alpha^2.$$
The first inequality above comes from the fact that the $\ell_2$ norm is subordinate to the $\ell_1$ norm, and the final inequality comes from (R1).  We conclude the following. 

\begin{theorem}\label{THMIcfbounded}
There exists a constant $C$, independent of $\alpha$, such that
  $$\|\F[\I_\alpha f]\|_{L_2(\R^d)}\leq C\dfrac{M_\alpha}{\gamma_\alpha}\|\F[f]\|_{L_2(Z)},\quad f\in PW_Z.$$
  \end{theorem}

 The next step toward the proof of Theorem \ref{THMmaintheorem} involves the definition of a multiplication operator $T_\alpha:L_2(Z)\to L_2(Z)$ defined by
$$T_\alpha(h) = \dfrac{\gamma_\alpha}{\phiahat}h.$$ The definition of $\gamma_\alpha$ implies that $\|T_\alpha\|\leq1$. We can rewrite Lemma \ref{LEMidentity} as
$$\F[f] = \F[\I_\alpha f]+\nsum{m}2^{dm}A_m^\ast\left(\dfrac{\phiahat(2^m\cdot)}{\gamma_\alpha}A_m\left(\gamma_\alpha\psi_\alpha\right)\right),$$ which by \eqref{EQpsidef} is
$$\F[f] = \F[\I_\alpha f]+ \dfrac{1}{\gamma_\alpha}\tau_\alpha\circ T_\alpha\left(\F[\I_\alpha f]\right) = \left(I+\frac{1}{\gamma_\alpha}\tau_\alpha\circ T_\alpha\right)\F[\I_\alpha f],$$
where $I$ is the identity operator on $L_2(Z)$.

\begin{proposition}\label{PROPinvertible}
 The map $I+\frac{1}{\gamma_\alpha}\tau_\alpha\circ T_\alpha$ is an invertible operator on $L_2(Z)$, and $$\left(I+\frac{1}{\gamma_\alpha}\tau_\alpha\circ T_\alpha\right)^{-1}\F[f] = \F[\I_\alpha f],\quad f\in PW_Z.$$
 Moreover,
 $$\left\|\left(I+\frac{1}{\gamma_\alpha}\tau_\alpha\circ T_\alpha\right)^{-1}\right\|\leq \dfrac{M_\alpha}{\gamma_\alpha}.$$
\end{proposition}

\begin{proof}
 Surjectivity of the operator in question follows from the identity in Lemma \ref{LEMidentity}.  To see injectivity, suppose that $(I+\frac{1}{\gamma_\alpha}\tau_\alpha\circ T_\alpha) h=0$ for some nonzero $h\in L_2(Z)$.  Let $f\in PW_Z$ be the function so that $\F[f]=h$.  Then by Lemma \ref{LEMidentity} and positivity of $\tau_\alpha$, we have 
 $$0 = \bracket{\left(I+\frac{1}{\gamma_\alpha}\tau_\alpha\circ T_\alpha\right) \F[f],T_\alpha\F[f]}_Z \geq\bracket{\F[f],T_\alpha\F[f]}_Z\geq0.$$
 Consequently, $T_\alpha\F[f]=0$, which implies that $\F[f]=0$ since $T_\alpha h=0$ if and only if $h=0$.
 
 Finally, the norm estimate follows from Theorem \ref{THMIcfbounded} and Lemma \ref{LEMidentity}.

\end{proof}

We now have the necessary ingredients to complete the proof.

\begin{proof}[Proof of Theorem \ref{THMmaintheorem}]
By Proposition \ref{PROPinvertible} and Lemma \ref{LEMidentity}, the following identity holds on $Z$:
$$\F[f]-\F[\I_\alpha f] = \left[I-\left(I+\frac{1}{\gamma_\alpha}\tau_\alpha\circ T_\alpha\right)^{-1}\right](\F[f]) = \left(I+\frac{1}{\gamma_\alpha}\tau_\alpha\circ T_\alpha\right)^{-1}\circ\frac{1}{\gamma_\alpha}\tau_\alpha\circ T_\alpha(\F[f]).$$

Therefore, if $f\in PW_{\beta B_2}$, Theorem \ref{THMIcfbounded} and Proposition \ref{PROPtauc} imply
\begin{align}\label{EQfirstratio}
  \|\F[f]-\F[\I_\alpha f]\|_{L_2(Z)} & \leq  \left\|\left(I+\frac{1}{\gamma_\alpha}\tau_\alpha\circ T_\alpha\right)^{-1}\right\|\frac{1}{\gamma_\alpha}\|\tau_\alpha\|\|T_\alpha\F[f]\|_{L_2(Z)} \nonumber\\
 &\leq  C\dfrac{M_\alpha}{\gamma_\alpha}M_\alpha\left\|\dfrac{1}{\phiahat(\cdot)}\F[f]\right\|_{L_2(\beta B_2)}\nonumber\\
 &\leq  C\dfrac{M_\alpha^2}{\gamma_\alpha m_\alpha(\beta)}\|\F[f]\|_{L_2(\beta B_2)}.
\end{align}

Next, we estimate $\|\F[\I_\alpha f]\|_{L_2(\R^d\setminus Z)}$ by familiar techniques: 

\begin{align}\label{EQsecondratios}
   \|\F[\I_\alpha f]\|_{L_2(\R^d\setminus Z)}^2 & \leq  \sum_{m\in\N}2^{dm}M_m(\alpha)^2\|A_m(\psi_\alpha)\|_{L_2(Z)}^2 \nonumber\\
  & \leq  5^dR_b^2S_\alpha^2\left\|\dfrac{1}{\phiahat(\cdot)}\F[\I_\alpha f]\right\|_{L_2(Z)}^2 \nonumber\\
  & \leq  C\dfrac{M_\alpha^2}{\gamma_\alpha^2}\|T_\alpha\F[\I_\alpha f]\|_{L_2(Z)}^2 \nonumber\\
  & \leq  C\left[\dfrac{M_\alpha}{\gamma_\alpha}\|T_\alpha\left(F[\I_\alpha f]-\F[f]\right)\|_{L_2(Z)}+\dfrac{M_\alpha}{\gamma_\alpha}\|T_\alpha\F[f]\|_{L_2(\beta B_2)}\right]^2 \nonumber\\
  & \leq  C\left[\dfrac{M_\alpha^3}{\gamma_\alpha^2m_\alpha(\beta)}+\dfrac{M_\alpha}{m_\alpha(\beta)}\right]^2\|\F[f]\|_{L_2(\beta B_2)}^2. 
\end{align}
The second inequality comes from \eqref{EQAmbound} and \eqref{EQpsidef}, while the next comes from the definition of $T_\alpha$ and condition (R1).  The final inequality is obtained from \eqref{EQfirstratio} and the fact that on $\beta B_2$, $\phiahat\geq m_\alpha(\beta)$ (therefore, $\|T_\alpha h\|_{L_2(\beta B_2)}\leq \frac{\gamma_\alpha}{m_\alpha(\beta)}\|h\|_{L_2(\beta B_2)}$).

Convergence of $\|\F[\I_\alpha f]-\F[f]\|_{L_2(\R^d)}$ depends ostensibly on the three ratios in \eqref{EQfirstratio} and \eqref{EQsecondratios}.  However, the largest is $\frac{M_\alpha^3}{\gamma_\alpha^2m_\alpha(\beta)}$.  Indeed one obtains this by multiplying $\frac{M_\alpha}{m_\alpha(\beta)}$ by $\frac{M_\alpha^2}{\gamma_\alpha^2}$ which is at least 1 by definition.  Similarly, $\frac{M_\alpha^2}{\gamma_\alpha m_\alpha(\beta)}\leq\frac{M_\alpha^2}{\gamma_\alpha m_\alpha(\beta)}\frac{M_\alpha}{\gamma_\alpha} = \frac{M_\alpha^3}{\gamma_\alpha^2m_\alpha(\beta)}$.  Consequently, if (R2) is satisfied, then $\|\F[\I_\alpha f]-\F[f]\|_{L_2(\R^d)}$ converges to 0 as $\alpha\to\infty$.

The moreover statement on the $L_2$ approximation rate now follows from the comments in the paragraph above and a straightforward combination of \eqref{EQfirstratio} and \eqref{EQsecondratios}.

To show uniform convergence on $\R^d$, use \eqref{EQfourierinversion} to see that
\begin{displaymath}
 \begin{array}{lll}
  |\I_\alpha f(x)-f(x)|& = & \dpifrac\left|\dint_Z\left(\F[\I_\alpha f](\xi)-\F[f](\xi)\right)e^{i\bracket{x,\xi}}d\xi+\dint_{\R^d\setminus Z}\F[\I_\alpha f](\xi)e^{i\bracket{x,\xi}}d\xi\right|\\
  & \leq & \dpifrac\left(\|\F[\I_\alpha f]-\F[f]\|_{L_1(Z)}+\|\F[\I_\alpha f]\|_{L_1(\R^d\setminus Z)}\right).\\
 \end{array}
\end{displaymath}

Now the first term, by the Cauchy-Schwarz inequality and \eqref{EQfirstratio} is at most a constant multiple of $\frac{M_\alpha^2}{\gamma_\alpha m_\alpha(\beta)}\|\F[f]\|_{L_2(\beta B_2)}$, whereas the second term may be estimated by a similar calculation to \eqref{EQsecondratios} and another appeal to the Cauchy-Schwarz inequality:
\begin{align*}\|\F[\I_\alpha f]\|_{L_1(\R^d\setminus Z)} & \leq \nsum{m}2^{dm}M_m(\alpha)m(Z)^\frac{1}{2}\|A_m(\psi_\alpha)\|_{L_2(Z)}\\ &\leq C\left(\dfrac{M_\alpha^3}{\gamma_\alpha^2m_\alpha(\beta)}+\dfrac{M_\alpha}{m_\alpha(\beta)}\right)\|\F[f]\|_{L_2(Z)}.\end{align*}
It follows that 
$$\|\I_\alpha f-f\|_{L_\infty(\R^d)}\leq C\dfrac{M_\alpha^3}{\gamma_\alpha^2m_\alpha(\beta)}\|\F[f]\|_{L_2},$$
and thus the proof is complete.
\end{proof}

\section{Remarks}\label{SECremarks}

\begin{remark}
 
  There are many ways in which one could choose to periodize the integrals over $\R^d\setminus Z$, and consequently, condition (I3) could well be formulated differently.  For example, if one periodizes using the annuli $jZ\setminus(j-1)Z$, then the condition would be that $(j^\frac{d}{2}M_j)\in\ell_1$, once the definition of $M_j$  is modified suitably.  However, this modification of $M_j$ essentially counteracts the change in annuli, and so does not give a substantially different condition.\end{remark}

\begin{remark}

As mentioned before, the problem of finding Riesz-basis sequences for $L_2(Z)$ is generally quite complex and depends heavily on the geometry of the set $Z$.  {\em Zonotopes} are Minkowski sums of line segments with one endpoint at the origin, and it is known (see, for example, \cite[Theorem 4.1.10]{gardner}) that there are zonotopes satisfying the hypothesis of Theorem \ref{THMmaintheorem}.  It was also shown by Lyubarskii and Rashkovskii \cite{lr} that for such a zonotope $Z$, $L_2(Z)$ has a Riesz-basis sequence. Their proof is for $d=2$, but the general statement for higher dimensions is alluded to there and also in \cite{bss}. Existence of Riesz-basis sequences for zonotopes higher dimensions was recently shown in \cite{grepstad_lev} and subsequently \cite{kol} under the assumption that the vertices of the zonotope lie on some lattice. \end{remark}

\begin{remark}

It is worth discussing the limiting case briefly.  All of Theorems \ref{THMGaussian}, \ref{THMmultiquadric}, and \ref{THMgeneralexample} hold in the case that we let $\delta=\beta=1$, which is the case that $Z=B_2$.  Indeed one needs only look at the end of the proof of Theorem \ref{THMmaintheorem} and see that the Dominated Convergence Theorem can be applied to show that $\inflim{\alpha}\|T_\alpha g\|_{L_2(B_2)} = 0$ for $g\in L_2(B_2)$.  However, as mentioned above, the result may be vacuous, because it is unknown if there is any Riesz-basis sequence for $L_2(B_2)$.  This is the primary reason for the analysis we have done here, to exploit the fact that we know there are Riesz-basis sequences for some convex bodies contained in the Euclidean ball.\end{remark}

\begin{remark}

For further reading on the interesting problem of finding Riesz-basis sequences, the reader is referred to \cite{ka,nitzan,pav,yo} for results in one dimension, and \cite{Bailey1,Bailey2,lr, chinese} for higher dimensions.
\end{remark}

\begin{remark} It is quite natural to ask if the condition that $X$ is a Riesz-basis sequence is a necessary one to obtain the type of recovery results mentioned here.  One relaxation to consider is the case when the exponential system $\left(e^{-i\bracket{x_j,\cdot}}\right)_{j\in\N}$ forms a {\em frame} rather than a Riesz basis.  A sequence $(h_j)_{j\in\N}$ forms a frame for a Hilbert space $\mathcal{H}$ if there are constants $A$ and $B$ such that $A\|f\|_\mathcal{H}^2\leq\sum_{j\in\N}|\bracket{f,h_j}_\mathcal{H}|^2\leq B\|f\|_\mathcal{H}^2$ for all $f\in\mathcal{H}$. Any function in $\mathcal{H}$ has an expansion of the form $\sum_{j\in\N}c_jh_j$, however frames differ from Riesz bases in that such expansions need not be unique.

First, we note that formation of the interpolant does not necessarily require the Riesz basis condition; indeed as long as the sequence $X$ satisfies \eqref{separationcondition}, a Gaussian interpolant, for example, of a bandlimited function can be formed.  However, as mentioned in Remark \ref{REMRiesz}, a key feature in the proofs we have provided is the uniqueness of the expansion in terms of the Riesz basis.  In the case where the exponentials merely form a frame, it is not necessarily true that $E(\psi_\alpha)=\Psi_\alpha$.  Moreover, the coefficients $(a_j)$ coming from the interpolation scheme cannot necessarily be represented as $\bracket{h,e^{-i\bracket{x_j,\cdot}}}$ for any $h\in L_2(Z)$.  Moreover, \eqref{EQAmbound} does not necessarily hold because only the upper bound of \eqref{EQRieszbasisconstant} holds for frames.

However, an extension of the ideas presented here which makes use of frames would be desirable, not only because they are more general (and more abundantly available), but also due to the fact that there are frames for many more general types of domains in $\R^d$, including the Euclidean ball.  Additionally, frames may be redundant, in that one allows multiple copies of the same vector in the frame, and so frame expansions lend themselves to being robust in the event of information loss (such as missing frame coefficients).

While we cannot at this time extend the main theorem here to the case where $X$ is a frame sequence, we may still make a minor extension in that direction, albeit one that does not involve interpolation.  
\begin{proposition}\label{PROPframe}
Let $Z$ be as before, and suppose that $X=\underset{n=1}{\overset{N}\bigcup}X_n$ where the $X_n$ are Riesz-basis sequences for $L_2(Z)$. Fix a family of $d$-dimensional regular interpolators that is regular for $PW_{\beta B_2}$.  Define the following approximation to $f\in PW(Z)$:
$$\widetilde{\I_\alpha}f(x):=\frac{1}{N}\finsum{n}{1}{N}\I_{\alpha}^{X_n}f(x),$$
where $\I_\alpha^{X_n}f$ is the interpolant of Theorem \ref{THMinterpolantbasic} which interpolates $f$ at the Riesz-basis sequence $X_n$.  Then $\inflim{\alpha}\widetilde{\I_\alpha}f=f$ both in $L_2(Z)$ and uniformly on $\R^d$.
\end{proposition}

The proof of the proposition is simply the triangle inequality and an appeal to Theorem \ref{THMmaintheorem}.  However, one should note that $\widetilde{\I_\alpha}f$ is not an interpolant to $f$ at $X$.  Also note that the union of Riesz-basis sequences need not be a Riesz-basis sequence (consider the case where $X_1=X_2$), so Proposition \ref{PROPframe} is indeed an extension of sorts.  However, even in the case where $X$ is the disjoint union of Riesz-basis sequences for $L_2(Z)$, it is unclear using the techniques of this paper whether or not one obtains convergence of the interpolant (which may be formed as in Theorem \ref{THMinterpolantbasic} as discussed above).
\end{remark}


\begin{thebibliography}{1}

\bibitem{AandS} M. Abramowitz and I. A. Stegun (Eds.), {\em Handbook of mathematical functions: with formulas, graphs, and mathematical tables.} No. 55, Courier Dover Publications, 1972.

\bibitem{geometry} S. Artstein-Avidan, A. Giannopoulos, and V. Milman, {\em Asymptotic Geometric Analysis, Part 1}, American Mathematical Society, 2015.

\bibitem{Bailey1}
B. A. Bailey, An asymptotic equivalence between two frame perturbation theorems, in {\em Proceedings of Approximation Theory XIII: San Antonio 2010}, Eds: M. Neamtu, L. Schumaker, Springer New York (2012), 1-7.

\bibitem{Bailey2}
B. A. Bailey, Sampling and recovery of multidimensional bandlimited functions via frames, {\em J. Math. Anal. Appl.} \textbf{370} (2010), 374-388. 

\bibitem{bss}
B. A. Bailey, Th. Schlumprecht, and N. Sivakumar, Nonuniform sampling and recovery of multidimensional bandlimited functions by Gaussian radial-basis functions, {\em J. Fourier Anal. Appl.}, \textbf{17}(3) (2011), 519-533.

\bibitem{gardner} R. J. Gardner, {\em Geometric Tomography, Second Edition}, Cambridge University Press, 2006.

\bibitem{grepstad_lev} S. Grepstad and N. Lev, Multi-tiling and Riesz bases, {\em Adv. Math.} \textbf{252} (2014), 1-6.

\bibitem{hamm} K. Hamm, Approximation rates for interpolation of Sobolev functions via Gaussians and allied functions, {\em J. Approx. Theory} \textbf{189} (2015), 101-122.

\bibitem{hmnw} T. Hangelbroek, W. Madych, F. Narcowich and J. Ward, Cardinal interpolation with Gaussian kernels, {\em J. Fourier Anal. Appl.} \textbf{18} (2012), 67-86.

\bibitem{hangelbroekron}  T. Hangelbroek, and A. Ron, Nonlinear approximation using Gaussian kernels, {\em J. Funct. Anal.} \textbf{259}(1) (2010), 203-219.

\bibitem{gardnerbulletin} R. Gardner, The Brunn-Minkowski inequality, {\em Bull. Amer. Math. Soc.} \textbf{39}(3) (2002), 355-405.

\bibitem{ka} M. I. Kadec, The exact value of the Paley-Wiener constant,  {\em Dokl. Adad. Nauk SSSR} \textbf{155} (1964), 1243-1254.

\bibitem{kol} M. N. Kolountzakis, Multiple lattice tiles and Riesz bases of exponentials, {\em Proc. Amer. Math. Soc.} \textbf{143} (2015), 741-747.

\bibitem{nitzan} G. Kozma, S. Nitzan, Combining Riesz bases, {\em Invent. Math.}, to appear.

\bibitem{ledford_scattered} J. Ledford, Recovery of Paley-Wiener functions using scattered translates of regular interpolators, {\em J. Approx. Theory} \textbf{173} (2013), 1-13.

\bibitem{ledford_bivariate}
J. Ledford, Recovery of bivariate band-limited functions using scattered translates of the Poisson kernel, {\em J. Approx. Theory} \textbf{189} (2015), 170-180.

\bibitem{lm} Yu. Lyubarskii and W. R. Madych, The recovery of irregularly sampled band limited functions via tempered splines, {\em J. Funct. Anal.} \textbf{125} (1994), 201-222.

\bibitem{lr} Yu. Lyubarskii and A. Rashkovskii, Complete interpolating sequences for Fourier transforms supported by convex symmetric polygons, {\em Ark. Math.} \textbf{38} (2000), 139-170.

\bibitem{pav} B. S. Pavlov, The basis property of a system of exponentials and the condition of Muckenhoupt, {\em Dokl. Akad. Nauk SSSR} \textbf{247} (1979), 37-40.

\bibitem{ss} Th. Schlumprecht  and N. Sivakumar, On the sampling and recovery of bandlimited functions via scattered translates of the Gaussian,  {\em J. Approx. Theory} \textbf{159} (2009), 128-153.

\bibitem{chinese} W. Sun, X. Zhou, On the stability of multivariate trigonometric systems, {\em J. Math. Anal. Appl.} \textbf{235} (1999), 159-167.

\bibitem{Wendland}
H. Wendland, {\em Scattered data approximation}, Vol. 17. Cambridge: Cambridge University Press, 2005.

\bibitem{yo} 	
R. M. Young,  {\em An Introduction to Nonharmonic Fourier Series, Revised Edition}, Academic Press, 2001.

\end{thebibliography}
\end{document}